\documentclass[12pt]{amsart}
\usepackage{amssymb,amsmath}
\usepackage[dvips]{graphicx}
\usepackage{pslatex}
\usepackage{amsmath,amscd}
\usepackage{amsthm}
\usepackage{overpic}
\usepackage{color}
\usepackage{stmaryrd}
\usepackage{pgf,tikz}
\usepackage{float}

\theoremstyle{plain} \newtheorem{thm}{Theorem}[section]
\newtheorem{cor}[thm]{Corollary} \newtheorem{prop}[thm]{Proposition}
\newtheorem{lemma}[thm]{Lemma}

\newtheorem*{namedtheorem}{\theoremname}
\newcommand{\theoremname}{testing}
\newenvironment{named}[1]{\renewcommand{\theoremname}{#1}\begin{namedtheorem}}{\end{namedtheorem}}

\theoremstyle{definition}

\theoremstyle{remark}

    \DeclareMathOperator{\Forb}{Forb}

 \usepackage{pgfplots}
\pgfplotsset{compat=1.15}
\usepackage{mathrsfs}
\usetikzlibrary{arrows}
\begin{document}

\author{Nicholas Abell}
\address{Department of Mathematics\\ 
Furman University\\ 
Greenville, SC 29613}
\email{abelni8@furman.edu}

\author{Elizabeth McDermott}
\address{Department of Mathematics\\ 
Furman University\\ 
Greenville, SC 29613}
\email{mcdeel3@furman.edu}

\author{Christian Millichap}
\address{Department of Mathematics\\ 
Furman University\\ 
Greenville, SC 29613}
\email{christian.millichap@furman.edu}

\title{Projective Planar Cartesian Products of Graphs}

\begin{abstract}
In this paper, we provide a complete classification of Cartesian products of graphs that embed in the projective plane. Our work requires us to determine minimal Cartesian products that are nonprojective planar, organize their essential properties to be used as constraints for projective planar embeddings, and explicitly construct projective planar embeddings for Cartesian products that satisfy these constraints. A corollary of our work shows that only six of the 35 forbidden minors for the projective plane are sufficient to classify projective planar Cartesian products. 
\end{abstract}

\maketitle

\section{Introduction}
\label{sec:Intro}

 A graph $G'$ is a \textbf{minor} of $G$ if $G'$ can be obtained by a sequence of vertex deletions, edge deletions, and edge contractions on $G$. The Graph Minor Theorem of Robertson and Seymour \cite{RoSe2004} shows that any minor closed graph property $\mathcal{P}$  exhibits a finite list of forbidden minimal minors. This means there is a finite minimal set of graphs $\Forb(\mathcal{P})$ such that a graph $G$ has property $\mathcal{P}$ if and only if $G$ has no minor in $\Forb(\mathcal{P})$. Furthermore, minimality here means that every $G \in \Forb(\mathcal{P})$ does not have $\mathcal{P}$ while every proper minor of $G$ does have $\mathcal{P}$; see \cite{Ma2017} for an introduction to minor closed properties. One such minor closed property is whether a graph $G$ \textbf{embeds} on a fixed compact connected surface $S$, that is, whether $G$ can be drawn on a surface $S$ where edges of $G$ don't cross each other. More formally,  an embedding is constructed via a continuous and one-to-one function $f: |G| \rightarrow S$, where $|G|$ represents a geometric realization of $G$; see \cite{GrTu2001} and \cite{MoTh2001} for further background on graph embeddings on surfaces.  As a result of the Graph Minor Theorem, the set $\Forb(S)$ of  forbidden minimal minors for embedding on $S$ must be finite. In this context, the Kuratowski-Wagner Theorem \cite{Ku1930}, \cite{Wa1937} can be restated as $\Forb(S^2) = \{ K_5, K_{3,3}\}$, where $S^{2}$ denotes the $2$-sphere. The only other surface with a known list of forbidden minors is the projective plane, $N_1$, where $\Forb(N_1)$ consists of $35$ elements and the completeness of this set was determined by Archdeacon in \cite{Ar1980}. We refer the reader to  Appendix A and Section 6.5 in Mohar and Thomassen \cite{MoTh2001} for this set of graphs. In this paper, we will represent the projective plane as an identification space constructed by identifying antipodal points on the boundary of a disk. Beyond the $2$-sphere and the projective plane, there currently isn't a complete list of $\Forb(S)$ for any other surface $S$. 

Since the sets $\Forb(S^2)$ and $\Forb(N_1)$ are explicitly determined, it is natural to classify when certain subclasses of graphs are either \textbf{planar} (embed in the plane or $S^2$) or \textbf{projective planar} (embed in the projective plane or $N_1$). Here, we are interested in the \textbf{Cartesian product of two graphs} $G$ and $H$, denoted $G \boxempty H$, which is defined to have vertex set $V(G) \times V(H)$ and a pair of vertices $(u,v), (u',v') \in V(G) \times V(H)$ determine an edge in $G \boxempty H$ if and only if $u = u'$ and $vv' \in E(H)$ or $uu' \in E(G)$ and $v=v'$. Such graphs admit a fiber structure that is frequently helpful for constructing embeddings and minors. A complete classification of planar Cartesian products was given by Behzad and Mahmoodian \cite{BeMa1969} and a proof can be found in \cite[Section 4.1]{ImKlRa2008}. Let $P_m$ represent a path on $m$ vertices and let $C_n$ represent a cycle on $n \geq 3$ vertices. Then their work shows that $G \boxempty H$ is planar if and only if $G \boxempty H$ is either a $P_m \boxempty P_n$, $P_m \boxempty C_n$, or $G \boxempty P_2$, where $G$ is \textbf{outerplanar}, meaning there exists an embedding of $G$ in the plane where all of the vertices of $G$ lie on the boundary of a single face.

 The straightforward classification of planar Cartesian products and the fact that $\Forb(N_1)$ is complete  motivates the need for a classification of Cartesian products of graphs that are projective planar. Our main theorem, which is stated below, provides such a classification. For visuals of graphs $B_{m}$, $Q_{m}$, $W_{m}$, and $X_{m}$, see Figure \ref{fig:Relevant graphs P3,C3 case}. We assume all graphs under consideration for this classification and the remainder of this paper are finite, simple, connected, and non-trivial (in the sense of being nonempty and not a single vertex), unless otherwise noted. 

\begin{thm}
\label{thm:maintheorem}
   The Cartesian product $G \boxempty H$ is projective planar if and only if either 
   \begin{enumerate}
       \item $G \boxempty H$ is planar,
       \item $G \boxempty H = K_{1,3} \boxempty K_{1,3}$,
       \item $G \boxempty H = G \boxempty P_3$ where $G$ is a minor of $B_{m}$ or $Q_{m}$ for some $m \in \mathbb{N}$, or
       \item $G \boxempty H = G \boxempty C_{3}$ where $G$ is a minor of $W_{m}$ or $X_{m}$ for some $m \in \mathbb{N}$. 
   \end{enumerate}
\end{thm}

Our work requires a careful analysis of when $G \boxempty H$ is projective planar for low complexity $H$.  In Section \ref{sec:ClassificationII}, we consider when $H = P_2$ and quickly see that $G \boxempty P_2$ is projective planar exactly when $G \boxempty P_2$ is planar. Most of our work goes into the cases where $|V(H)| = 3$, which are covered in Section \ref{sec:ClassificationIII}. In these scenarios, there are infinite classes of $G \boxempty H$ that are projective planar and nonplanar. This work first requires us to determine minimal Cartesian products that are nonprojective planar and organize their essential properties. Then, we show that any $G \boxempty H$ that doesn't have these properties must meet the qualifications of (3) and (4) from Theorem \ref{thm:maintheorem}, and we construct projective planar embeddings for such graphs. A similar but less complicated analysis goes into Section \ref{sec:ClassificationIV}, where we consider $G \boxempty H$ where each factor has at least four vertices. In this scenario, $K_{1,3} \boxempty K_{1,3}$ is the only projective planar Cartesian product that is nonplanar.  

The proof of Theorem \ref{thm:maintheorem} also provides explicit lists of sufficient forbidden  minors for when a Cartesian product of graphs embeds in the projective plane. One list consists of the six graphs from the list of $35$ forbidden minors in $\Forb(N_1)$ that are sufficient for our classification; see Corollary \ref{cor:MMNPP} and Appendix \ref{sec:Appendix}. The other set consists of nine Cartesian products that serve as a sufficient set of forbidden minors to classify projective planar Cartesian products; see Corollary \ref{cor:MMNPP2}. Throughout the upcoming sections of the paper, we will reference the six graphs of $\Forb(N_1)$ used in our work via the labels given in Appendix \ref{sec:Appendix}, which correspond with the labels given in \cite{MoTh2001}.

As noted above, our work  requires several explicit constructions of embeddings and minors of Cartesian products of graphs. For these constructions, it frequently helps to use the fiber decomposition of $G \boxempty H = (G \times V(H)) \cup (V(G) \times H)$ where $G \times V(H)= \bigcup_{h\in V(H)} (G\times \{h\})$ and every $G \times \{h\}$ is  called a \textbf{G-fiber}. Likewise, $V(G) \times H$ is a union of H-fibers.  This characterization says that a Cartesian product of $G$ and $H$ is the union of its G-fibers and H-fibers. Furthermore, any two distinct G-fibers are disjoint, and similarly for H-fibers. While this characterization is not a disjoint union, it is easy to see that the intersection of any G-fiber with any H-fiber is a single vertex in $G \boxempty H$. We refer the reader to \cite{ImKlRa2008} for further background on graph Cartesian products. 
 
We would like to acknowledge the Furman University Mathematics Department for financially supporting this work via the Summer Mathematics Undergraduate Research Fellowships. We also would like to thank Doug Rall and Thomas Mattman for their helpful feedback on a draft of this paper.


\section{Classification Part I: one factor has two vertices}
\label{sec:ClassificationII}

If $H$ is simple and connected with two vertices, then  $H = P_2$. In this case, the classification of projective planar $G \boxempty P_2$ is the exact same as the planar case. 



\begin{thm}
    \label{thm:N1 embedding P2 Case}
    Let $G$ be a graph. Then the following statements are equivalent:
    \begin{enumerate}
        \item $G$ does not contain a $K_{4}$ or $K_{2,3}$ minor. 
        \item $G$ is outerplanar.
        \item $G \boxempty P_2$ is planar.
        \item $G \boxempty P_2$ is projective planar.
    \end{enumerate}
\end{thm}

\begin{proof}
The equivalence of (1) and (2) follows from the work of Chartrand and Harary in \cite{ChHa1967}. The equivalence of (2) and (3) is given in Proposition 4.3 of \cite{ImKlRa2008} and follows from the classification of planar Cartesian products, which was originally proved in \cite{BeMa1969}. 

It suffices to show that (4)  implies one of (1)-(3) and vice versa to complete the proof. First, suppose $G$  does contain a $K_{4}$ or $K_{2,3}$ minor. Note that, $K_4 \boxempty P_2$ is isomorphic to  $\mathcal D_{17} \in \Forb(N_1)$ and $K_{2,3} \boxempty P_2$ contains a subgraph isomorphic to  $\mathcal{G}_1 \in \Forb(N_1)$. Thus, $G \boxempty P_2$ is nonprojective planar, showing (4) implies (1). Now if we suppose $G \boxempty P_2$ is planar, then clearly $G \boxempty P_2$ is also projective planar, showing (3) implies (4), and completing the proof. 
\end{proof}

\section{Classification Part II: one factor has three vertices}
\label{sec:ClassificationIII}

We now consider the cases where $G \boxempty H$ is projective planar with $|V(H)|=3$. Since $H$ is simple and connected, this implies that $H$ is either $P_3$ or $C_3$. We first set some notation and prove some lemmas that will be used in both of these cases. The graphs in Figure \ref{fig:Relevant graphs P3,C3 case} will frequently be referenced in this section. Along the top row, there are four infinite classes of graphs, $X_m,B_m,W_m,$ and $Q_m$ with $m\in\mathbb N$. At each of the dotted edges, the index $m$ denotes the number of edges between the endpoints. 
\begin{figure}
    \centering
    \begin{overpic}[width=0.8\linewidth]{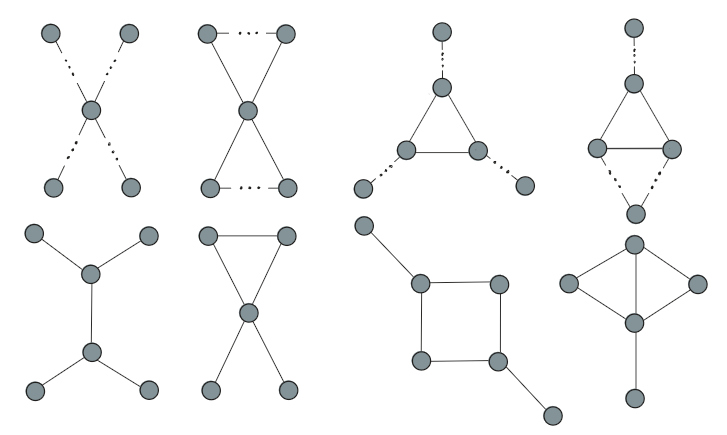}
    \put(1,55){$X_m$}
    \put(22.5,55){$B_m$}
    \put(54.5,55){$W_m$}
    \put(81,55){$Q_m$}
    \put(1.5,26.5){$I$}
    \put(82.5,25.5){$L$}
    \put(25,26.5){$R$}
    \put(46,27.5){$S$}
    \end{overpic}
    \caption{Relevant graphs for constructions in Section \ref{sec:ClassificationIII}. The subscript $m$ for $X_m$, $B_m$, $W_m$, and $Q_m$ indicates a $P_{m+1}$ subgraph in the dotted lines.}
    \label{fig:Relevant graphs P3,C3 case}
   
\end{figure}
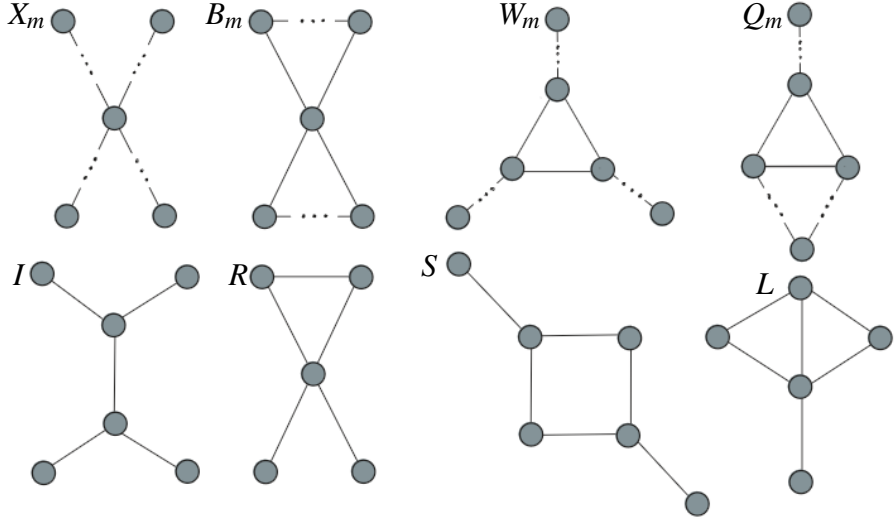


\begin{lemma}
    \label{lem:I box P3 nonembeding in N1}
    The Cartesian product $I \boxempty P_3$ does not embed on $N_1$. 
\end{lemma}
\begin{proof}

    We claim that $I \boxempty P_3$ has a $\mathcal{G}_1$ forbidden minor for the projective plane. We obtain a $\mathcal{G}_{1}$ minor from $I \boxempty P_3$ by first deleting the four $P_3$ fiber edges incident to the two vertices of degree four in $I \boxempty P_3$. Then we edge contract all of the remaining edges in the $P_3$ fibers. This process yields a pair of disjoint $K_{2,3}$ graphs connected by three edges in the desired manner to form $\mathcal{G}_{1}$. Thus $I \boxempty P_3$ does not embed on $N_1$.
\end{proof}

\begin{lemma}
    \label{lem:K1,5 box P3 nonembedding in N1}
    The Cartesian product $K_{1,5} \boxempty P_3$ does not embed on $N_1$.
\end{lemma}
\begin{proof}

    We claim that $K_{1,5} \boxempty P_3$ has a $K_{3,5} = \mathcal{E}_{3}$ forbidden minor for the projective plane. We obtain a $K_{3,5}$ minor from $K_{1,5} \boxempty P_3$ by first deleting the two $P_3$ fiber edges incident to the unique vertex of degree seven in $K_{1,5} \boxempty P_3$. Then, we edge contract along all remaining edges from the $P_3$ fibers, yielding a $K_{3,5}$ minor. Thus, $K_{1,5} \boxempty P_3$ does not embed on $N_1$.
\end{proof}

Note that, Lemma \ref{lem:K1,5 box P3 nonembedding in N1} implies that if $G \boxempty P_3$ embeds on $N_1$ and $|V(H)|=3$, then for all $v\in V(G)$, $\deg(v)\le 4$. 

\begin{lemma}
    \label{lem:S box P3 nonembeding in N1}
   The Cartesian product $S \boxempty P_3$ does not embed on $N_1$. 
\end{lemma}
\begin{proof}
     We claim that $S \boxempty P_3$ contains an $\mathcal{E}_{22}$ forbidden minor for the projective plane. We assign vertex labels with a coordinate system where the first coordinate represents an $S$ fiber and the second coordinate denotes a $P_3$ fiber. We label the vertices contained in the 4-cycle of an $S$ fiber $1$ through $4$ in a clockwise manner and label the vertex 5 which connects to vertex 1 and the vertex 6 which connects to vertex 3. We label the vertices on a $P_3$ fiber $1$ through $3$ in the usual way. We will provide labels for new vertices resulting from edge contraction while all other vertices retain their labels from the previous step. We construct this minor by first deleting the edges $[(1,2),(1,3)]$ and $[(3,2),(3,3)]$. Next we edge contract along the $P_3$ fiber between (5,1), (5,2), and (5,3), forming a new vertex $v_5$. Likewise we edge contract along the $P_3$ fiber between (6,1), (6,2), and (6,3), forming one new vertex, $v_6$. Next we edge contract along the four remaining edges in the third $S$ fiber and the two edges connecting this $S$ fiber to $v_5$ and $v_6$, respectively. This sequence of edge contractions forms a single $v_7$. Finally, we delete the two edges $[(1,2), v_7]$ and $[(3,2), v_7]$. Then we have an $\mathcal{E}_{22}$ forbidden minor, as needed. 
\end{proof}

\begin{lemma}
    \label{lem:G is a tree P3,C3 case}
    Suppose $G \boxempty H$ embeds in $N_1$ and $G$ is a tree. If $|V(H)|=3$, then $G$ is a minor of $X_{m}$.
\end{lemma} 
\begin{proof}
    Suppose $G\boxempty H$ embeds in $N_1$ and $G$ is a tree. By Lemma \ref{lem:K1,5 box P3 nonembedding in N1},  we have $\deg(v)\le 4$, for all  $v \in V(G)$. In addition,  Lemma \ref{lem:I box P3 nonembeding in N1} shows that $G$ cannot contain an  $I$ minor, and so, there exists at most one vertex $v_0 \in V(G)$ with $\deg(v_0)=3$ or $\deg(v_0)=4$. Suppose for all $v \in V(G)$, we have $\deg(v)\le 2$. Then since $G$ is a tree, $G$ must be a path, which is clearly a subgraph of $X_m$ for an appropriate choice of $m \in \mathbb{N}$.  Otherwise, $G$ is a tree where there exists  a unique $v_{0} \in V(G)$ with $\deg(v_0)=3$ or $\deg(v_0)=4$, while for all $v \in V(G) - \{v_0\}$, we have $\deg(v) \leq 2$. In either case,  $G$ is a subgraph of $X_m$ for sufficiently large $m \in \mathbb{N}$, completing the proof.    
\end{proof}

\begin{lemma}
    \label{lem:G is pseudotree - cycles P3,C3 case}
    Suppose $G \boxempty P_3$ embeds in $N_1$ and $G$ contains at least two distinct cycles, $C$ and $C'$. Then $C \cap C'\neq \emptyset$.
\end{lemma}
\begin{proof}
    Suppose $G \boxempty P_3$ embeds in $N_1$ and $G$ contains at least two distinct cycles, say $C$ and $C'$. Then since $G$ is connected, there exists a path in $G$ connecting a vertex $v_0 \in V(C)$ to a vertex $v_1 \in V(C')$. Without lose of generality, assume this is the shortest path in $G$ between these two cycles. As a result, this path only intersects $C$ in $\{v_0\}$ and it only intersects $C'$ in $\{v_1\}$. Edge contract along this path until $v_0$ and $v_1$ are connected by a single edge. Then $G$ has an $I$ minor. By Lemma \ref{lem:I box P3 nonembeding in N1}, G does not embed on $N_1$, giving a contradiction. Thus, $C \cap C' \neq \emptyset$.
\end{proof}

\subsection{One factor is $P_3$}
\label{subsec:P3}
In this section, we assume $H= P_3$ and $G$ is a simple, connected graph. Here, we classify exactly which $G \boxempty P_3$ embed in $N_1$. 

\begin{lemma}
    \label{lem:Demise of Evil Lollipop}
    The Cartesian product $L \boxempty P_3$ does not embed on $N_1$.
\end{lemma}
\begin{proof}
    We  claim $L \boxempty P_3$ contains an $\mathcal{E}_5$ forbidden minor for the projective plane. We assign vertex labels with a coordinate system where the first coordinate represents an $L$ fiber and the second coordinate denotes a $P_3$ fiber. We label the vertices contained in the 4-cycle of the $L$ fiber 1 through 4 in a clockwise manner so that label 3 applies to the unique degree four vertex in $L$. The final vertex of $L$ is labeled $5$. We label the vertices on a $P_3$ fiber 1 through 3 in the usual way. We will provide new labels for  vertices resulting from edge contraction while all other vertices retain their labels from the previous step. We construct this minor by first deleting the two $P_3$ fiber edges between vertices $(3,1), (3,2),$ and $(3,3)$. Then we delete the edges $[(1,2),(2,2)]$ and $[(1,2),(4,2)]$. Next, we delete the edges $[(1,1),(3,1)]$ and $[(1,3),(3,3)]$. Finally, we preform three sets of edge contractions. First we edge contract the edges between $(5,1), (5,2)$, and $(5,3)$ into $v_5$. Then we contract the two edges between $(2,1), (2,2),$ and $(2,3)$ into $v_2$, and lastly we contract the two edges between $(4,1)$, $(4,2)$, and $(4,3)$ into $v_4$. This results in an $\mathcal{E}_5$ forbidden minor, as needed. 
\end{proof}

\begin{lemma}
Let $G$ be a graph with two edge distinct cycles $C$ and $C'$, and suppose that $V(C\cap C')$ contains at least two distinct vertices. Then there exists three internally disjoint paths between two vertices $v_1, v_2 \in V(C \cap C')$ in $C \cup C'$.
    \label{lem:CycleIntersectionMinor}
\end{lemma}
\begin{proof}
   Let $C,C'$ be distinct cycles in $G$ and suppose that $V(C\cap C')$ contains at least two distinct vertices. As $C$ and $C'$ are distinct, there exists $e_1\in E(C)-E(C')$ that is incident to some $v_1 \in V(C \cap C')$. Since $v_1\in V(C')$, we have that $v_1$ is incident to two distinct edges in $C'$. Since $e_1\not\in E(C')$, $e_1$ is distinct from these edges. Hence, $\deg(v_1)\ge 3$. Now begin a walk in $C$ starting at $v_1$  and first moving along $e_1$. Terminate this walk at the first vertex you visit in $V(C\cap C')$, which we label as $v_2$. By the same logic as with $v_1$ case, we also see that $\deg(v_2)\ge3$. By construction, our walk determines a path $p_1$  that is contained in $(C-C') \cup \{v_1,v_2\}$. Clearly there exists two internally disjoint paths from $v_1$ to $v_2$ along edges in $C'$, which we denote $p_2$ and $p_3$ respectively. Since path $p_1$ lies in $C-C'$ except for its endpoints, it is internally disjoint from $p_2$ and $p_3$. Hence, $v_1$ and $v_2$ are connected by three internally disjoint paths in $C \cup C'$. 
\end{proof}

\begin{figure}
    \begin{overpic}[width=0.9\linewidth]{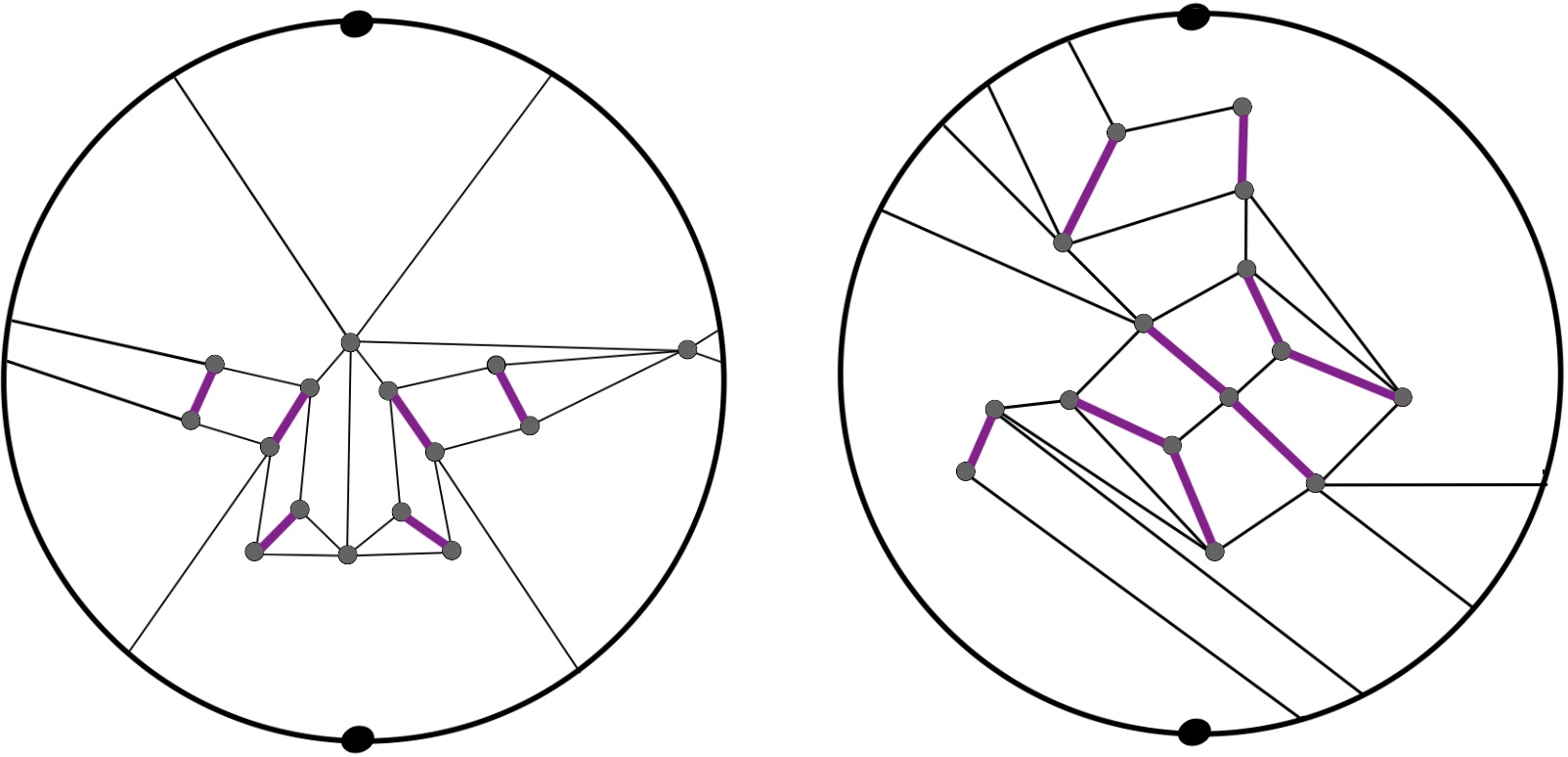}
    \put(18.5,49){$B_1\boxempty P_3$}
    \put(72,49.5){$Q_1\boxempty P_3$}
    \put(-2,25.6){\textbf{b}}
    \put(46.5,25.5){\textbf{b'}}
    \put(51.5,25.5){\textbf{b}}
    \put(100,25.5){\textbf{b'}}
     \put(23.5,27){\tiny{$(3,2)$}}
     \put(24,24.5){\tiny{$(4,2)$}}
     \put(27,20.5){\tiny{$(5,2)$}}
     \put(32.5,23.5){\tiny{$(4,3)$}}
     \put(34,20){\tiny{$(5,3)$}}
     \put(41,27){\tiny{$(3,3)$}}
     \put(12,26){\tiny{$(2,3)$}}
     \put(7,20){\tiny{$(1,3)$}}
     \put(17,25){\tiny{$(2,2)$}}
     \put(12,19){\tiny{$(1,2)$}}
     \put(14,11){\tiny{$(1,1)$}}
     \put(20,11){\tiny{$(3,1)$}}
     \put(27,11){\tiny{$(5,1)$}}
     \put(17,17){\tiny{$(2,1)$}}
     \put(23,17){\tiny{$(4,1)$}}

     \put(63,32){\tiny{$(2,2)$}}
     \put(71,41){\tiny{$(1,2)$}}
     \put(78,42.5){\tiny{$(1,1)$}}
     \put(80,36){\tiny{$(2,1)$}}
     \put(75,32){\tiny{$(5,1)$}}
     \put(90,23){\tiny{$(3,1)$}}
     \put(80.5,23.5){\tiny{$(4,1)$}}
     \put(73,22.5){\tiny{$(4,2)$}}
     \put(75.5,19){\tiny{$(4,3)$}}
     \put(85.5,18){\tiny{$(3,2)$}}
     \put(78,12){\tiny{$(3,3)$}}
     \put(65,24){\tiny{$(5,3)$}}
     \put(58.5,22){\tiny{$(2,3)$}}
     \put(56.5,17.5){\tiny{$(1,3)$}}
     \put(71.5,29){\tiny{$(5,2)$}}
    \end{overpic}
    \caption{Projective planar embeddings of $B_1\boxempty P_3$ and $Q_1 \boxempty P_3$.}
    \label{fig:N1 embeddings B,Q}
\end{figure}

\begin{thm}
    \label{thm:N1 embedding P3 Case}
    $G \boxempty P_3$ embeds on $N_1$ if and only if $G$ is a minor of $B_m$ or $Q_m$ for $m \in \mathbb{N}$.
\end{thm}

\begin{proof}
    We proceed in the backwards direction first, constructing $B_m$ and $Q_m$, respectively. 
    
     We first consider $B_m$, for some $m\in \mathbb{N}$. We construct the projective planar embedding of the base case $B_1 \boxempty P_3$ in Figure \ref{fig:N1 embeddings B,Q}. Now we inductively show how to build $B_m \boxempty P_3$ from the embedding of $B_1 \boxempty P_3$ in Figure \ref{fig:N1 embeddings B,Q}. We have highlighted the edges $[(1,1), (2,1)]$, $[(1,2),(2,2)],$ and $[(1,3),(2,3)]$ in purple which are contained in a $P_2 \boxempty P_3$ subgraph of $G$. Note that the edge $[(1,2), (2,2)]$ bounds two quadrilateral faces, one with $[(1,1), (2,1)]$ and one with $[(1,3),(2,3)]$. Thus we can subdivide each highlighted edge forming vertices $v_1, v_2,$ and $v_3$, respectively. Then we embed a $P_3$ fiber connecting $v_1, v_2,$ and $v_3$ in these quadrilateral faces. Likewise, we perform the same operation in the face bound by the edges $[(4,1),(5,1)], [(4,2),(5,2)]$, and $[(4,3),(5,3)]$ forming the vertices $v_4, v_5,$ and $v_6$ resulting in a projective planar embedding of $B_2 \boxempty P_3$. To build a projective planar embedding of $B_m \boxempty P_3$, we work with the two $P_2 \boxempty P_3$ subgraphs of this new embedding whose boundary contains $(1,1), v_1, (1,2), v_2, (1,3), v_3$ and $(4,1), v_4, (4,2), v_5, (4,3), v_6$ and repeat the same process. 
     
     We now consider $Q_m$ for some $m \in \mathbb{N}$. Consider the projective planar embedding of $Q_1 \boxempty P_3$ in Figure \ref{fig:N1 embeddings B,Q} as our base case. Note that, the index $m$ determines the construction of a $C_{2m+1}$ subgraph and a  $P_{m+1}$ subgraph in $Q_m$ that are disjoint; we deal with these cases separately to build $Q_m \boxempty P_3$. First consider the edges $[(3,1),(4,1)], [(3,2),(4,2)],$ and $[(3,3),(4,3)]$ highlighted in purple which are contained in a $P_2 \boxempty P_3$ subgraph of $Q_1 \boxempty P_3$. We use the same technique as in the $B_m$ case: since the edges are bound in two quadrilateral faces, subdivide each highlighted edge and connect the new vertices with an embedded $P_3$ fiber. Likewise perform the same operation on the edges $[(4,1),(5,1)], [(4,2),(5,2)],$ and $[(4,3),(5,3)]$. This shows how to construct the $C_{2m+1}$ subgraph inductively, so we now consider the $P_m$ subgraph. Note that the edges $[(1,1),(2,1)], [(1,2),(2,2)],$ and $[(1,3),(2,3)]$ are similarly found in a $P_2 \boxempty P_3$ subgraph bounding two quadrilateral faces, so we apply the same technique. This results in a projective planar embedding of $Q_2 \boxempty P_3$ and we can proceed inductively again.\\

    We now consider the forward direction, and so, we suppose $G \boxempty P_3$ is projective planar. We will show that $G$ is either a minor of $B_m$ or $Q_m$ for some $m \in \mathbb{N}$ or that $G \boxempty P_3$ is nonprojective planar, contradicting our assumption. By Theorem \ref{thm:N1 embedding P2 Case}, Lemma \ref{lem:I box P3 nonembeding in N1}, Lemma \ref{lem:K1,5 box P3 nonembedding in N1}, Lemma \ref{lem:S box P3 nonembeding in N1}, and Lemma \ref{lem:Demise of Evil Lollipop}, it suffices to show that $G$ is either a minor of $B_m$ or $Q_m$ for some $m \in \mathbb{N}$ or $G$ contains a $K_{4}$, $K_{2,3}$, $I$, $K_{1,5}$, $S$, or $L$ minor (collectively, forbidden minors for $G$), which contradicts $G \boxempty H$ being projective planar. 
    
    We consider three cases: either $G$ is a tree, $G$ has exactly one cycle, or $G$ has at least two distinct cycles. 
    
    \underline{Case 1:} If  $G$ is a tree, then by Lemma \ref{lem:G is a tree P3,C3 case}, $G$ is a minor of $X_m$ for some $m \in \mathbb{N}$, which is a subgraph of $B_m$. 

    \underline{Case 2:} Suppose $G$ contains exactly one cycle, $C$.  We break this case into sub-cases based on whether $C = C_3$ or $C = C_n$ for $n \geq 4$.
    
    Now, suppose $C = C_3$ is the only cycle in $G$. If $G$ is not a minor of $B_m$, then $E(G) -E(C) \neq \emptyset$.  First, suppose there exists $v_0 \in V(C)$ such that there are at least three edges in $E(G)-E(C)$ incident to $v_0$. Then $G$ contains a $K_{1,5}$ subgraph, which is a contradiction. Now, suppose there exists $v_{1} \in V(C)$ such that there are exactly two distinct edges, $e_1$ and $e_2$, in $E(G)-E(C)$ incident to $v_1$. If $G$ is not a minor of $B_m$, then there exists an edge $e_3$ in $E(G)-E(C)$ such that $e_3 \neq e_i$ for $i=1,2$ and $e_3$ is incident to either some $v_2 \in V(C)$ with $v_2 \neq v_1$ or $e_3$ is incident to a degree three or greater vertex on some subpath of $G - E(C)$ that contains either $e_1$ or $e_2$. In the former case, $G$ contains an $I$ subgraph with $v_1$ and $v_2$ corresponding to the degree three vertices in this $I$ subgraph, which gives a contradiction. In the latter case, $G$ again contains an $I$ subgraph. At this point, we can assume each $v \in V(C)$ has at most one edge in $E(G)-E(C)$ incident to it and at least one vertex of $V(C)$ is incident to such an edge. Consider the components of the forest $G-E(C)$. By the previous statements, there are one, two, or  three such components and at most one of these components is incident to each vertex of $V(C)$. If any of these components contains a degree three vertex, then $G$ contains an $I$ minor. So we can assume each such tree in $G - E(C)$ only has vertices of degree at most two. Thus, each component is a path, and there is at most one such path connected to each vertex in $C$. In any of these cases, $G$ is a minor of $Q_m$ for some $m \in \mathbb{N}$.  

    Now, suppose $G$ has exactly one cycle $C = C_n$, for some $n \geq 4$. If $G$ is not a minor of any $Q_m$ for $m \in \mathbb{N}$, then we see that $E(G) - E(C) \neq \emptyset$. Following the same argument as was done in the previous case,  we see that for each $v \in V(C)$ there is at most one edge in $E(G)-E(C)$ incident to $v$. Furthermore,  at least one $v \in V(C)$ must be incident to some edge in $E(G)-E(C)$ since $E(G) - E(C) \neq \emptyset$ and $G$ is connected. Again, similar to the previous case we consider the nonempty forest $G - E(C)$. Each component of $G - E(C)$ must be a path, since otherwise, $G$ contains an $I$ minor. If $G-E(C)$ contains two or more path components, then $G$ either contains  an $I$ minor or $S$ minor, depending on whether these components are incident to adjacent or nonadjacent vertices, respectively, in $V(C)$. Otherwise, $G-E(C)$ must have a single path component, and so, $G$ is a minor of $Q_m$ for some $m \in \mathbb{N}$.  

    \underline{Case 3:} Now, suppose $G$ has at least two distinct cycles, $C$ and $C'$. We consider two sub-cases: either $C \cap C'$ contains at least two distinct vertices or $C \cap C' = \{v_0\}$ for some $v_0 \in V(G)$. 
    
    \underline{Case A:} First,  suppose $C \cap C' = \{v_0\}$, for some $v_0 \in V(G)$. If $G$ is not a minor of $B_m$,  then $E(G)-E(C \cup C')$ is nonempty. Furthermore, since $G$ is connected, there must exist $e \in E(G) -E(C \cup C')$ such that at least one vertex of $e$ is incident to some vertex in $V(C \cup C')$. If $e$ is incident to $v_0$, then $\deg(v_0) \geq 5$, and so, $G$ contains a $K_{1,5}$ minor. If $e$ is incident to only one vertex in $V(C \cup C') - \{v_{0}\}$, then $G$ contains an $I$ minor.  Lastly, suppose $e$ is incident to two distinct vertices in $V(C \cup C')  - \{v_{0}\}$. Then $G$ contains an $L$ minor.

   \underline{Case B:}  Now suppose $C \cap C'$ contains at least two distinct vertices. First, suppose the shortest path in $C \cup C'$ between any two vertices in $C \cap C'$ contains at least two edges. Then Lemma \ref{lem:CycleIntersectionMinor} shows that there exists $v_1, v_2 \in V(C \cap C')$ such that there are three internally disjoint paths between $v_1$ and $v_2$ in $C \cup C'$. Since we assumed any path from $v_1$ to $v_2$ can not be a single edge, these three internally disjoint paths in $C \cup C'$ determine a $K_{2,3}$ minor in $C \cup C'$ where $v_1$ and $v_2$ represent the degree three vertices for this minor. This shows $G$ has a $K_{2,3}$ minor, and we are done in this case. So now suppose the shortest path in $C \cup C'$ between two vertices $v_{1}, v_{2} \in C \cap C'$ is exactly one edge. We first claim that there exist cycles $C''$ and $C'''$ in $G$ such that $C'' \cap C'''$ is exactly one edge. First, note that $C$ and $C'$ could meet these qualifications. If they don't, then Lemma \ref{lem:CycleIntersectionMinor} provides three internally disjoint paths  in $C \cup C'$ between some $v, v' \in V(C \cap C')$. One of these paths must be a single edge, $e$, since otherwise, $C \cup C'$ has a $K_{2,3}$ minor and we are done. Let $P$ and $P'$ represent the other internally disjoint paths between $v$ and $v'$. Then $C'' = P \cup \{e\}$ and $C''' = P' \cup \{e\}$ works. 
   
   Moving forward, we relabel so that $C$ and $C'$ are the cycles in $G$ such that $C \cap C'$ is precisely one edge $e = (v_{1}, v_{2})$ of $G$.   We again need to consider sub-cases: either  both $C$ and $C'$ are not $3$-cycles or at least one of these cycles is a $3$-cycle.

   \underline{Case i:} Suppose $C = C_n$ and $C' = C_{l}$ for $n,l \geq 4$. Since $C_n$ and $C_l$ intersect along an edge whose endpoints are $v_{1},v_{2} \in V(G)$, we see that $C \cup C'$ contains an $I$ subgraph where $v_{1}$ and $v_{2}$ are the degree three vertices for this $I$ subgraph.

   \underline{Case ii:} Suppose $C = C_3$ and $C' = C_{n}$ for $n \geq 3$. If $G = C \cup C'$, then $G$ is a minor of $Q_{m}$ for some $m \in \mathbb{N}$. So, we assume $G - (C \cup C') \neq \emptyset$. Since $G$ is connected, this implies that there exists some $e' \in E(G) - E(C \cup C')$ that is incident to $C \cup C'$. First, suppose $e'$ is incident to two vertices in $C \cup C'$. Since $G$ is a simple graph, there are only a few distinct cases to consider for such $e'$. If $C' = C_3$, then there is only one possibility for $e'$, specifically, $e'$ is incident to the two vertices in $V(C \cup C') - \{v_1, v_2\}$. In this case, $G$ admits a $K_4$ minor. If $C' = C_{n}$ for $n \geq 4$, then there are two possibilities: either $e'$ is incident to one vertex in $\{v_{1}, v_{2}\}$ and one vertex in $V(C') - \{v_{1}, v_{2}\}$ or $e'$ is incident to two vertices in $V(C') - \{v_{1}, v_{2}\}$. In either case, it is not hard to see $C \cup C' \cup \{e'\}$ contains an $L$ minor. Moving forward, we no longer need to consider any edges incident to two vertices in $V(C \cup C')$. So, suppose $e'$ is incident to only one vertex in $V(C \cup C')$. Consider $H = C \cup C' \cup \{e'\}$. Let $\{v_{3}\} = V(C) - \{v_1, v_2\}$. If $e'$ is incident to $v_3$, then $H$ is a subgraph of some $Q_m$. Likewise, if $C' = C_3$ and $e'$ is incident to $v \in V(C') - \{v_{1},v_{2}\}$, then $H$ is a subgraph of $Q_m$. In all other cases, $H$ contains an $L$ minor, and so, the same holds for $G$. If $G - H \neq \emptyset$, then there exists edge $e'' \in E(G) - E(H)$ that is incident to $H$. First, if $C' = C_{3}$, then $e''$ could be incident to the unique degree two vertex in $H$, which is the one scenario unique to $C' = C_{3}$. However, in this case $H$ has an $S$-minor, and so, $G$ has an $S$ minor. By the previous arguments, if $G$ does not contain a forbidden minor, then there are considerable restrictions on the incidence structure of $e''$ in $G$: $e''$ must be incident to one vertex of $e'$ and the other vertex of $e''$ must be disjoint from $H$. Consider the subgraph $H' = C \cup C' \cup \{e', e''\}$. Then there are two possibilities for the structure of $H$. In one case $H'$ is a minor of some $Q_{m}$ and in the other  case $H'$ contains an $I$ minor. We now only need to consider what happens if $G - H' \neq \emptyset$ where $H' = C \cup C' \cup \{e', e''\}$ and $H'$ is a subgraph of some $Q_m$. By the previous restrictions on additional edges (to make sure we don't we have a forbidden minor), we see that  $G-E(H')$ could only be a path connected to the degree one vertex of $e''$, and so, $G$ must be a minor of some $Q_m$.  
\end{proof}

\subsection{One factor is $C_3$} In this section, we assume $H= C_3$ and $G$ is a simple, connected graph. Here, we classify exactly which $G \boxempty C_3$ embed in $N_1$.
\label{subsec:C3}

\begin{lemma}
    \label{lem: R box C3 nonembedding in N1}
    $R \boxempty C_3$ does not embed on $N_1$.
\end{lemma}
\begin{proof}
    We claim that $R \boxempty C_3$ contains a $\mathcal G_1$ forbidden minor for the projective plane. Note that we can decompose $\mathcal{G}_1$ into $(K_{2,3} \boxempty P_2) -2e$ where we delete the two edges connecting pairs of degree four vertices in disjoint $K_{2,3}$ fibers. We use this decomposition to construct the $\mathcal{G}_1$ minor. We assign vertex labels based on a coordinate system where the first coordinate represents position in an $R$ fiber and the second coordinate denotes a $C_3$ fiber. For the $R$ fibers we label the vertices in the $C_3$ subgraph 1-3 in a clockwise manner. We label the remaining two vertices $4$ and $5$ and assume $3$ is the label for the unique degree four vertex in an $R$ fiber. Likewise, we label the vertices on the $C_3$ fibers 1,2, and 3 in a clockwise manner. We will provide labels for new vertices resulting from edge contraction while all other vertices retained their labels from the previous step. We construct the $\mathcal{G}_1$ minor by first deleting the $C_3$ fiber  edges between vertices $(3,1), (3,2),$ and $(3,3)$. Next we edge contract the three edges on the $C_3$ fiber between vertices $(4,1), (4,2)$, and $(4,3)$ and label this single vertex $v_4$. Likewise, we edge contract the three edges between the vertices $(5,1), (5,2)$, and $(5,3)$, labeling this vertex $v_5$. Note that, the vertices $v_4$, $v_5$, $(3,1), (3,2),$ and $(3,3)$ induce  a $K_{2,3}$ subgraph in this minor. None of our forthcoming minor operations will affect this subgraph. Then we delete the edge between $[(2,1),(2,2)]$ and $[(1,1), (1,3)]$. Now we choose to edge contract vertices $(1,1)$ and $(1,2)$ into $v_1$.  Finally, we delete the edges $[v_1,(3,1)]$, $[v_1,(3,2)]$, and $[(2,3),(3,3)]$. This gives us a second copy of $K_{2,3}$ with $v_1$ and $(2,3)$ in the set of two vertices and $(2,1), (2,2)$, and $(1,3)$ in the set of three vertices. Since we have the three edges $[(2,1),(3,1)]$, $[(2,2),(3,2)]$, and $[(1,3),(3,3)]$ connecting the two copies of $K_{2,3}$, we have formed a $\mathcal{G}_1$ minor of $R \boxempty C_3$. Since $\mathcal{G}_1 \in$ Forb($N_1$), $R \boxempty C_3$ does not embed on the projective plane.
\end{proof}

\begin{lemma}
\label{lem:C4 box C3 nonembedding in N1}
    $C_4 \boxempty C_3$ is nonprojective planar.
\end{lemma}
\begin{proof}
We will show that $C_4 \boxempty C_3$ contains a $E_{22}$ minor and since $E_{22} \in \Forb(N_1)$,  this will complete the proof. Take one of the $C_4$ fibers of $C_4 \boxempty C_3$  and edge contract along all four of its edges. This reduces this $C_4$ fiber to a single vertex, $v_0$, of degree eight with an edge connecting $v_0$ to each of the remaining eight vertices. Deleting four of the edges incident to $v_0$ then creates the necessary $E_{22}$ minor.   
\end{proof}

\begin{lemma}
    \label{lem: G psuedotree - cycles C3 case}
    If $G \boxempty C_3$ is projective planar, then $G$ contains at most one cycle, and if such a cycle exists in $G$, then it must be a 3-cycle.
\end{lemma}
\begin{proof}
Suppose $G \boxempty C_3$ embeds on $N_1$. Lemma \ref{lem:C4 box C3 nonembedding in N1} shows that $C_4\boxempty C_3$ does not embed in $N_1$, and so, it follows that $G$ does not contain an $n$-cycle for any $n \geq 4$. Thus, we only consider the existence of $3$-cycles in $G$. Now suppose for contradiction that $G$ contains at least two distinct $3$-cycles, say $C$ and $C'$. By Lemma \ref{lem:G is pseudotree - cycles P3,C3 case}, we know $C \cap C' \neq \emptyset$. As $G$ is simple, $C \cap C'$ must either be a single vertex or a single edge. If $C \cap C'$ is a single vertex in $G$, then $G$ contains an $R$ subgraph obtained by deleting an edge from $C \cup C'$. If $C \cap C'$ is an edge $e$, then $G$ has a $C_4$ subgraph obtained by edge deleting $e$ from $C \cup C'$. By Lemma \ref{lem: R box C3 nonembedding in N1} and Lemma \ref{lem:C4 box C3 nonembedding in N1}, $G \boxempty C_3$ do not embed in $N_1$, giving a contradiction. Therefore, $G$ has at most one cycle, which, if it exists, is a 3-cycle.
\end{proof}

\begin{figure}
\centering
\begin{overpic}[width=0.8\linewidth]{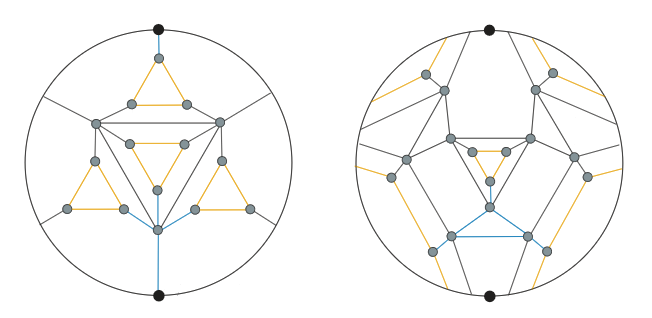}
    \put(19,47){$X_1\boxempty C_3$}
    \put(69.5,47){$W_1\boxempty C_3$}
    \put(1,23){\textbf{b}}
    \put(46,23){\textbf{b'}}
    \put(52,23){\textbf{b}}
    \put(96.5,23){\textbf{b'}}
    \end{overpic}
    \caption{Projective planar embeddings of $X_1\boxempty C_3$ and $W_1\boxempty C_3$.}
    \label{fig:N1 embeddings X,W}
\end{figure}
\begin{thm}
    \label{thm:N1 embedding C3 Case}
    $G \boxempty C_3$ embeds on $N_1$ if and only if $G$ is a minor of $W_m$ or $X_{m}$ for some $m \in \mathbb{N}$.
\end{thm}
\begin{proof}
    We cover the backwards direction with two constructive cases. 
    
    \underline{Case 1:} Let $G = X_{m}$, for some $m \in \mathbb{N}$. Consider the projective planar embedding of $X_1\boxempty C_3$ in Figure \ref{fig:N1 embeddings X,W}. The blue edges denote an $X_1$-fiber, while the orange edges represent four distinct $C_3$ fibers, each bounding a distinct triangular face in this embedding. We will inductively show how to build $X_m \boxempty C_3$ from the embedding  of $X_1\boxempty C_3$ given in Figure \ref{fig:N1 embeddings X,W}. Let $C_{T}$ represent the $3$-cycle that bounds triangle $T$ for one of the four orange triangles. Add an embedded $3$-cycle $C_{T}'$ to the interior of $T$. Embed three additional edges in $T$ so that each vertex of $C_{T}$ is incident to one and only one vertex of $C_{T}$. Clearly, this can be done for all four orange triangles in a manner that still gives a projective planar embedding, which will be an embedding for $X_2 \boxempty C_3$. To build a projective planar embedding of $X_3 \boxempty C_3$, one now works with the four triangles bound by the collection of four $C_{T}'$ cycles from the previous step and then repeats the same process.  
    
    \underline{Case 2:} Let $G = W_{m}$, for some $m \in \mathbb{N}$. Consider the projective planar embedding of $W_1\boxempty C_3$ in Figure \ref{fig:N1 embeddings X,W}. The blue edges denote a $W_1$-fiber, while the orange edges represent three distinct $C_3$ fibers, each bounding a distinct triangular face in this embedding. We use the same technique as in the $X_m$ case: in the interior of each of these three orange triangles, embed a $3$-cycle and embed a set of three additional edges connecting one vertex of this new $3$-cycle to one and only one vertex of the $3$-cycle that bounds this orange triangle. This results in a projective planar embedding of $W_{2} \boxempty C_{3}$ and one can clearly proceed inductively again. 

For the forward direction,  assume $G \boxempty C_3$ embeds in $N_1$. We will show that $G \boxempty C_3$ must be a minor of either $W_m$ or $X_{m}$ for some $m \in \mathbb{N}$. By Lemma \ref{lem: G psuedotree - cycles C3 case}, either $G$ is a tree or $G$ contains only one cycle which must be a $3$-cycle. If $G$ is a tree, then by Lemma \ref{lem:G is a tree P3,C3 case}, $G$ is a minor of $X_{m}$ for some $m \in \mathbb{N}$. So, suppose $G$ contains a single $3$-cycle, $C$, and no other cycles. First, if $G-C = \emptyset$, then $G$ is a $3$-cycle, which is a subgraph of $W_1$, and we are done. So assume $G-C \neq \emptyset$. Then $G - C$ must be a forest. Edge contract $C$ along its three edges, forming vertex $v_0$, to obtain a tree $G'$ that is a minor of $G$. Thus by Lemma \ref{lem:G is a tree P3,C3 case}, $G'$ is a minor of $X_m$.  This immediately tells us that each component of $G'-\{v_0\}$ is a path. Clearly $G'-\{v_0\}=G-C$, so $G-C$ must also be a disjoint collection of paths; let $\mathcal{P}$ designate this set of components of $G-C$. We now consider how elements of $\mathcal{P}$ connect to $C$ in order to form $G$. Suppose there exist two (distinct) elements in $\mathcal{P}$ that are connected at the same vertex of $C$, each by a distinct edge that was deleted from $G$ in forming $G-C$. If this happens, then $G$ clearly contains an $R$ minor, and so, $G \boxempty C_3$ does not embed in $N_1$ by Lemma \ref{lem: R box C3 nonembedding in N1}, giving a contradiction. Thus, there is at most one element of $\mathcal{P}$ attached to each vertex of $C$ via an edge in order to construct $G$. In this case, $G$ is a subgraph of $W_m$ for an appropriate choice of $m \in \mathbb{N}$.
\end{proof}



\section{Classification Part III: both factors have at least four vertices}
\label{sec:ClassificationIV}

We now classify when $G \boxempty H$ embeds in $N_1$ where $|V(G)|, |V(H)| \geq 4$. We first prove some essential facts about simple connected graphs with at least four vertices. Here, $C_{3} *_{v} P_{2}$ represents the amalgamation of the graphs $C_3$ and $P_{2}$ along a single vertex from each such graph. 

\begin{lemma}\label{lem:fourvertexminors}
Let $G$ be a graph with $|V(G)| \geq 4$. Then $G$ contains one of the following graphs as a minor: $P_4$, $C_4$, $K_{1,3}$, $C_{3} *_{v} P_{2}$, $K_{4} - e$, or $K_4$.
\end{lemma}

\begin{proof}
  First, suppose $G$ has exactly four vertices. Since $G$ is a simple connected graph, $G$ must be a minor of $K_4$. One can check that the simple connected minors of $K_4$ with exactly four vertices are exactly the six graphs on our list. If $G$ has $|V(G)| = n > 4$, then one can contract along some edge of $G$ and then delete any parallel edges created from this edge contraction to obtain a simple connected graph $G'$ where $|V(G')| = n-1$ and $G'$ is a minor of $G$. By inducting on $n$, one can obtain a simple connected minor of $G$ with exactly four vertices, and so, this minor must be one of the six graphs on our list, as needed.     
\end{proof} 

\begin{lemma}
    \label{lem:P4 box K1,3 nonembedding in N1}
    $P_4 \boxempty K_{1,3}$ does not embed on $N_1$.
\end{lemma}
\begin{proof}
    We claim $P_{4} \boxempty K_{1,3}$ contains a $\mathcal{G}_1$ forbidden minor for the projective plane. We denote the $K_{1,3}$ fibers adjacent to only one other copy of $K_{1,3}$ as ``external" whereas ``internal" $K_{1,3}$ fibers are adjacent to two distinct $K_{1,3}$ fibers. First, delete the three edges between vertices of degree four or five in $P_4 \boxempty K_{1,3}$ to obtain a minor $G'$. Edge contract along all six of the edges of $G'$ that connect a degree two vertex of an external $K_{1,3}$ fiber to a degree three vertex on an internal fiber. The resulting minor of  $P_4 \boxempty K_{1,3}$ is $\mathcal{G}_1$.
\end{proof}


\begin{lemma}
    \label{lem:K1,4 box K1,3 doesn't embed}
    $K_{1,4} \boxempty K_{1,3}$ does not embed on $N_1$.
\end{lemma}
\begin{proof}
 We claim $K_{1,4} \boxempty K_{1,3}$ contains an $E_{18} = K_{4,4} - \{e\}$ forbidden minor for the projective plane. Let $v_{0}$ be the unique degree seven vertex of $K_{1,4} \boxempty K_{1,3}$. We call the $K_{1,3}$ fiber that contains $v_{0}$ the central $K_{1,3}$ fiber. To obtain our desired minor,  first delete all three edges in the central $K_{1,3}$ fiber. Then we edge contract each of the other $K_{1,3}$ fibers so that each becomes a single vertex. This produces a $K_{4,4}$ minor. We then delete one edge to obtain a $K_{4,4} - \{e\} = E_{18}$ minor, as desired.   
\end{proof}

\begin{lemma}\label{lem:G has K1,4 or P4 subgraph}
Suppose $G$ is not a path or cycle and $|V(G)| \geq 5$. Then $G$ contains either a $P_4$ subgraph or a $K_{1,4}$ subgraph.
\end{lemma}

\begin{proof}
    Let G be a graph meeting these qualifications. Since $G$ is not a path or cycle, there exists $v_0 \in V(G)$ such that $\deg(v_0) \ge 3$ where $v_0$ is connected by distinct edges to three other distinct vertices, say $v_1, v_2, v_3$. Let $H$ be the subgraph of $G$ consisting of these four vertices and three edges. Note, $H$ is isomorphic to $K_{1,3}$. Since $|V(G)| \geq 5$, the set $V(G) - V(H)$ is nonempty. Since $G$ is connected, some vertex in  $V(G) - V(H)$, say $v_4$, must connect to a vertex of $H$ via a path in $G$. If $v_4$ connects to $v_i$ for $i=1,2,3$, then $G$ contains a $P_4$ subgraph. Otherwise, $v_4$ must connect to $v_0$, and so, $G$ contains a $K_{1,4}$ subgraph.
\end{proof}

There are some  useful subgraph inclusions from the set of six graphs in Lemma \ref{lem:fourvertexminors} which will be used in the proof of the following theorem. We use $H \leq G$ to denote $H$ is a subgraph of $G$. Specifically, we have 
\begin{enumerate}
    \item $P_4 \leq C_4 \leq K_4 - e \leq K_4$
    \item $P_4 \leq C_{3} *_{v} P_{2} \leq K_{4} - e \leq K_{4}$
    \item $K_{1,3} \leq C_{3} *_{v} P_{2} \leq K_{4} - e \leq K_{4}$
\end{enumerate}

In addition, any of these six graphs contains either a $P_4$ subgraph or a $K_{1,3}$ subgraph. 

\begin{thm}\label{thm:N1 embedding 4 Vertex Case}
     Suppose $G$ and $H$ are graphs with  $|V(G)|, |V(H)| \geq 4$. Then $G \boxempty H$ is projective planar if and only if $G \boxempty H$ is planar or $K_{1,3} \boxempty K_{1,3}$. 
\end{thm}

\begin{proof}
    Let $G$ and $H$ be graphs with $|V(G)|, |V(H)| \ge 4$. We first proceed in the backwards direction.  If $G \boxempty H$ is planar, then $G \boxempty H$ is also projective planar. We provide a construction of $G \boxempty H = K_{1,3} \boxempty K_{1,3}$ in Figure \ref{fig:N1 embedding K1,3 box K1,3}.This completes the backwards direction. 
    \begin{figure}
    \centering
    \begin{overpic}
        [scale=0.18]{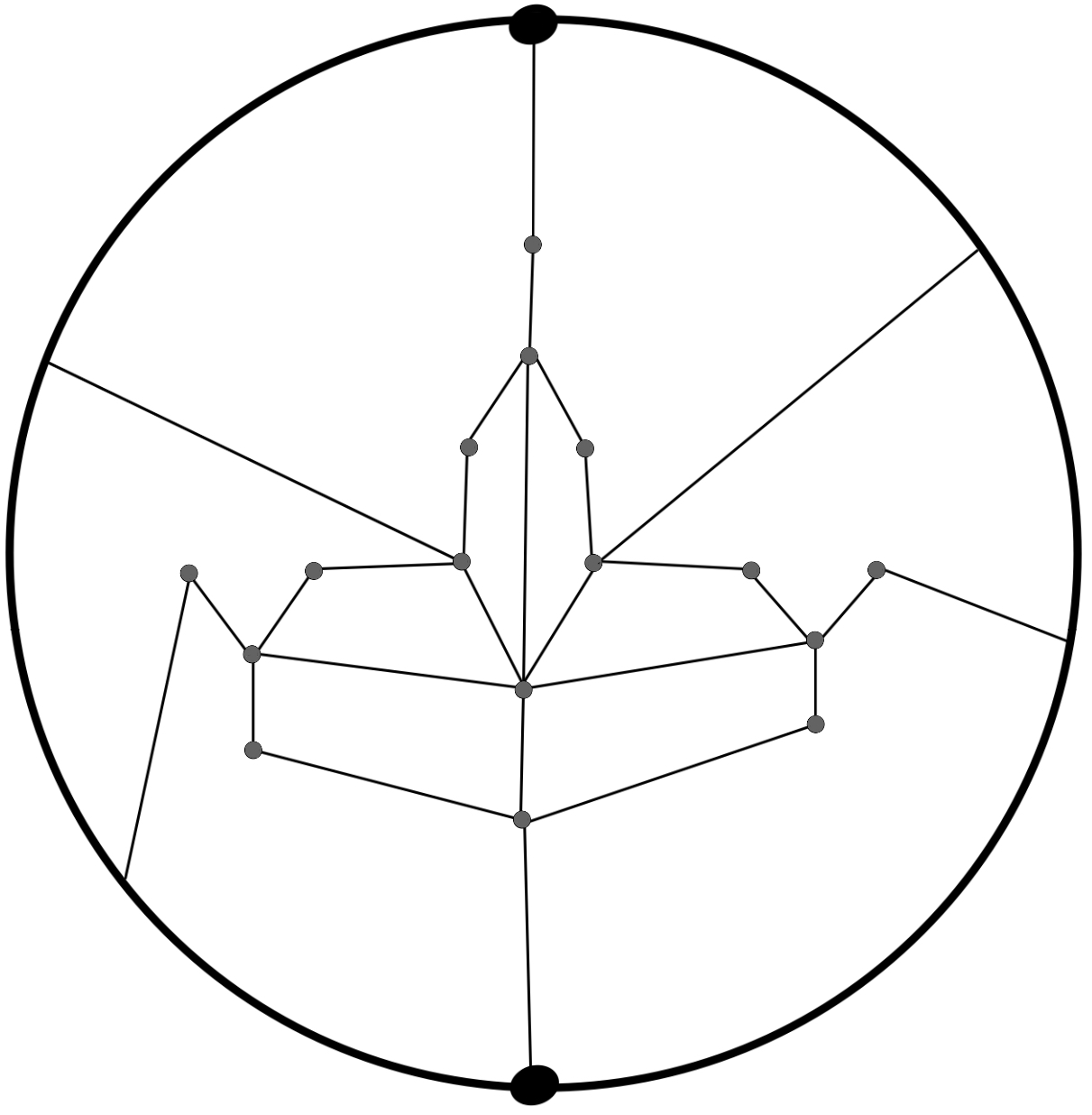}
			\put(12,50){\tiny{$(2,1)$}}
            \put(23,50){\tiny{$(2,2)$}}
            \put(23,37){\tiny{$(2,3)$}}
            \put(18,29){\tiny{$(2,4)$}}
            \put(34,58){\tiny{$(3,2)$}}
            \put(53,58){\tiny{$(3,1)$}}
            \put(48,68){\tiny{$(3,3)$}}
            \put(40,75){\tiny{$(3,4)$}}
            \put(34.5,46){\tiny{$(1,2)$}}
            \put(53.5,46){\tiny{$(1,1)$}}
            \put(47,35){\tiny{$(1,3)$}}
            \put(39,23){\tiny{$(1,4)$}}
            \put(64,50){\tiny{$(4,1)$}}
            \put(75,50){\tiny{$(4,2)$}}
            \put(73.5,40){\tiny{$(4,3)$}}
            \put(72,32){\tiny{$(4,4)$}}
            \put(-3,45){\textbf{b}}
            \put(98,45){\textbf{b'}}          
    \end{overpic}
    \caption{Projective Plane Embedding of $K_{1,3} \boxempty K_{1,3}.$}
        \label{fig:N1 embedding K1,3 box K1,3}
    \end{figure}

Now, assume  $G\boxempty H$ is projective planar. We will show $G \boxempty H$ must be either planar or $K_{1,3} \boxempty K_{1,3}$. We consider two cases: either $|V(G)|, |V(H)|=4$ or, without loss of generality, $|V(G)|\ge 5$.

  First, suppose $|V(G)|, |V(H)| = 4$.  Note that, $ P_4 \boxempty K_{1,3}$ is nonprojective planar by Lemma \ref{lem:P4 box K1,3 nonembedding in N1}. Thus, by considering (1) and (2) from our subgraph inclusion lists, we have that $ G \boxempty K_{1,3}$ is projective planar only when $G= K_{1,3}$. We now consider when neither $G$ nor $H$ has a $K_{1,3}$ minor. Then $G$ and $H$ must be either $P_4$ or $C_4$. Both $P_4 \boxempty P_4$ and $P_4 \boxempty C_4$ are planar. By Lemma \ref{lem:C4 box C3 nonembedding in N1}, we know that $C_4 \boxempty C_4$ is nonprojective planar. Thus if $|V(G)|, |V(H)|=4$, $G \boxempty H$ is planar or $K_{1,3} \boxempty K_{1,3}$, as needed.

Now, suppose $|V(G)|\ge 5$ and $|V(H)|\ge 4$. We consider the following sub-cases: G is a path, G is a cycle,  or G is neither a path nor a cycle.\\
    \indent \underline{Case 1:} Suppose G is a path. Since $|V(G)| \geq 5$, we see that $G$ contains a $P_4$ subgraph. From Lemma \ref{lem:P4 box K1,3 nonembedding in N1}, we have that H does not contain a $K_{1,3}$ minor. Therefore, $H$ is either $P_n$ or $C_n$, and in either case,  $G \boxempty H$ is planar.\\
    \indent \underline{Case 2:} Suppose $G$ is a cycle. Since $|V(G)| \geq 5$, we see that $G$ must contain a $P_4$ minor. Similar to Case 1, we see that H cannot contain a $K_{1,3}$ minor, and so, H must be a path or cycle. If H is a cycle, then $G \boxempty H$ contains a $C_4 \boxempty C_3$ minor. Then by Lemma \ref{lem:C4 box C3 nonembedding in N1}, $G \boxempty H$ is nonprojective planar. Thus, H must be a path, which makes $G \boxempty H$  planar.\\
    \indent \underline{Case 3:} Suppose $G$ is neither a path nor a cycle. 
    Then by Lemma \ref{lem:G has K1,4 or P4 subgraph}, we have that $G$ contains a $P_4$ or $K_{1,4}$ subgraph. In addition, since $G$ is neither a path nor a cycle, there must exist $v \in V(G)$ such that $\deg(v) \geq 3$ which implies $G$ contains a $K_{1,3}$ subgraph. By our subgraph inclusion lists for graphs with four vertices, $H$ must contain a  $P_4$ or $K_{1,3}$ minor. For any of these possible cases, $G \boxempty H$ contains either a $P_4 \boxempty K_{1,3}$ minor or a $K_{1,4}  \boxempty K_{1,3}$ minor. By Lemma \ref{lem:P4 box K1,3 nonembedding in N1} and Lemma \ref{lem:K1,4 box K1,3 doesn't embed}, we see that  $G \boxempty H$ is nonprojective planar in these cases. 
\end{proof}


\section{Forbidden minimal minor sets for Cartesian products}

\label{sec:Forbiddensets}

At this point, we can justify Theorem \ref{thm:maintheorem},  which we restate here.
\begin{named}{Theorem \ref{thm:maintheorem}}
The Cartesian product $G \boxempty H$ is projective planar if and only if either 
   \begin{enumerate}
    \item $G \boxempty H$ is planar,
    \item $G \boxempty H = K_{1,3} \boxempty K_{1,3}$,
    \item $G \boxempty H = G \boxempty P_3$ where $G$ is a minor of $B_{m}$ or $Q_{m}$ for some $m \in \mathbb{N}$, or
    \item $G \boxempty H = G \boxempty C_{3}$ where $G$ is a minor of $W_{m}$ or $X_{m}$ for some $m \in \mathbb{N}$. 
   \end{enumerate}
\end{named}

\begin{proof}
    The backwards direction follows from explicit constructions given throughout this paper. See Figure \ref{fig:N1 embeddings B,Q}, Figure \ref{fig:N1 embeddings X,W}, and Figure \ref{fig:N1 embedding K1,3 box K1,3} and their corresponding descriptions.

    For the forwards direction, suppose $G \boxempty H$ is projective planar. Theorem \ref{thm:N1 embedding P2 Case} shows that if $H$ has exactly two vertices, then $G \boxempty H$ is planar. Theorem \ref{thm:N1 embedding P3 Case} and Theorem \ref{thm:N1 embedding C3 Case} cover the cases where $H$ has exactly three vertices, which are classified by the graphs described in items (3) and (4). By commutativity of Cartesian products, it remains to cover the case where both $G$ and $H$ have at least four vertices. In this case,  Theorem \ref{thm:N1 embedding 4 Vertex Case} shows that $G \boxempty H$ is either planar or $K_{1,3} \boxempty K_{1,3}$.
\end{proof}

As noted in Section \ref{sec:Intro}, the work of Robertson-Seymour shows that any property closed under taking minors exhibits a finite list of  forbidden minimal minors. Furthermore, the work of Archdeacon \cite{Ar1980} shows that one only needs to consider whether a graph $G$ contains one of the $35$ elements in $Forb(N_1)$ as a minor to determine if $G$ is projective planar. Our work shows that if we restrict our analysis to  Cartesian products of graphs, then it is sufficient to consider only six elements of $Forb(N_1)$; see  Appendix \ref{sec:Appendix} for visuals of these six graphs.

\begin{cor}
    \label{cor:MMNPP}
The set $\{ \mathcal{D}_{17}, \mathcal{E}_{3}, \mathcal{E}_{5}, \mathcal{E}_{18}, \mathcal{E}_{22}, \mathcal{G}_{1}  \} \subset \Forb(N_1)$ is a sufficient set of forbidden minors to classify projective planar Cartesian products.
\end{cor}

We could also determine a set of Cartesian products of graphs that serve as a sufficient list of forbidden minors for Cartesian products to embed on a fixed surface. We first consider the planar case, and we determine a necessary and sufficient list in this case. Sufficiency is a corollary of the classification of planar Cartesian products given by Behzad and Mahmoodian \cite{BeMa1969}. Here, we use $H \preceq G$ to designate that graph $H$ is a minor of graph $G$. 

\begin{prop}\label{prop:ForbiddenCPforplane}
  The set  $\Forb_{CP}(S^2) = \{ C_3 \boxempty C_3, K_{1,3} \boxempty P_3, K_{2,3} \boxempty P_2, K_4 \boxempty P_2\}$ is a necessary and sufficient set of forbidden Cartesian product minors to classify  planar Cartesian products. 
  \end{prop}
\begin{proof}
Observe that $C_3 \boxempty C_3,K_{1,3}\boxempty P_3$ and $K_{2,3} \boxempty P_2$ each contain a $K_{3,3}$ minor, and $K_{4} \boxempty P_2$ contains a $K_5$ minor. Hence, they are all nonplanar. Next, we need to justify this is a sufficient list of forbidden minors. That is, if $G \boxempty H$ is nonplanar then $G \boxempty H$ contains a minor from this list. This follows from the proof given in \cite{ImKlRa2008} of the classification of planar Cartesian products. First, if both $G$ and $H$ only have vertices of degree one or two then $G \boxempty H$ is either planar or contains a $C_3 \boxempty C_3$ minor. Now suppose one factor has a vertex of at least degree three, say $G$. Then $H$ either contains a $P_3$ subgraph or $H= P_2$. If $H$ contains a $P_3$ subgraph, then $G \boxempty H$ contains a $K_{1,3} \boxempty P_3$ subgraph. If $H = P_2$, then $G \boxempty P_2$ is planar if and only if $G$ is outerplanar, which holds if and only if $G$ does not contain a $K_4$ or $K_{2,3}$ minor.

We now justify this set is necessary. Note that all four of these graphs are distinct since they each have a distinct number of vertices. It suffices to show that if $G, G' \in \Forb_{CP}(S^2)$ and $G \neq G'$, then $G' \not \preceq G$. Observe that any minor operation applied to a connected graph removes at least one edge, while either reducing or preserving the vertex count. These observations provide two immediate constraints on minors of connected graphs. If $G'$ is a proper minor of $G$,  then $|V(G)| \geq |V(G')|$, and $|E(G)| > |E(G')|$. As a result of these constraints, we only need to check that $K_4 \boxempty P_2 \not\preceq C_3 \boxempty C_3$, $K_4 \boxempty P_2 \not\preceq K_{2,3} \boxempty P_2$, and $K_4\boxempty P_2\not\preceq K_{1,3}\boxempty P_3$. 

We first suppose that $K_4 \boxempty P_2 \preceq C_3 \boxempty C_3$ and show this leads to a contradiction. Notice that each of these graphs is 4-regular, $|V(C_3 \boxempty C_3)|-|V(K_4 \boxempty P_2)|=1$, and $|E(C_3 \boxempty C_3)|-|E(K_4 \boxempty P_2)|=2$. Thus, we must perform exactly one edge contraction or exactly one vertex deletion (and possibly an edge deletion in either case) on $C_3 \boxempty C_3$  to obtain a $K_4 \boxempty P_2$ minor. If we edge contract, we create a vertex $v_1$ with $\deg(v_1)=6$ and we reduce the edge count by one. Even with an edge deletion, $v_1$ will have a degree of at least five, so this minor of $C_3\boxempty C_3$ cannot be 4-regular. Performing these operations in reverse still creates the same issue. Instead, if we delete a vertex and possibly delete an edge (or vice versa) of $C_3 \boxempty C_3$, then we remove at least four edges from $C_3\boxempty C_3$. Thus, the minor of $C_3 \boxempty C_3$ that we created has at most $14$ edges, but $|E(K_4 \boxempty P_2)| = 16$, which is a contradiction. Therefore, we have that $K_4 \boxempty P_2 \not\preceq C_3 \boxempty C_3$. 

For the other two cases, suppose $K_4 \boxempty P_2 \preceq K_{2,3} \boxempty P_2$ and $K_4\boxempty P_2\preceq K_{1,3}\boxempty P_3$. Since $|E(K_{2,3} \boxempty P_2)| - |E(K_4 \boxempty P_2)| =1$ and $|E(K_{1,3} \boxempty P_3)| - |E(K_4 \boxempty P_2)| =1$, we see that only one minor operation would be used in both cases. We now note that if  $G'$ is a minor of $G$  (where both graphs are connected) and only one minor operation is used to construct $G'$ from $G$, then $|E(G)|-|E(G')|\ge|V(G)|-|V(G')|$. This inequality follows from our earlier observations on vertex count and edge count under minor operations and the fact that our graphs are connected. However, in both of our cases (which only require one minor operation on connected graphs), this inequality is violated. Thus, we must have $K_4 \boxempty P_2 \not\preceq K_{2,3} \boxempty P_2$, and $K_4\boxempty P_2\not\preceq K_{1,3}\boxempty P_3$, as needed.
\end{proof}

By combining Proposition \ref{prop:ForbiddenCPforplane} with Theorem \ref{thm:maintheorem} we can provide a succinct characterization of  Cartesian products of graphs that are projective planar but not planar, i.e.,  Cartesian products of graphs with crosscap number one. 


\begin{cor} \label{cor:CrosscapOne}
    The Cartesian product $G \boxempty H$ has crosscap number one if and only if either 
   \begin{enumerate}
    \item $G \boxempty H = K_{1,3} \boxempty K_{1,3}$,
    \item $G \boxempty H = G \boxempty P_3$ where $G$ is a minor of $B_{m}$ or $Q_{m}$ for some $m \in \mathbb{N}$ and $G \boxempty P_3$  contains a minor from $\Forb_{CP}(S^2)$, or
    \item $G \boxempty H = G \boxempty C_{3}$ where $G$ is a minor of $W_{m}$ or $X_{m}$ for some $m \in \mathbb{N}$ and $G \boxempty C_{3}$ contains a minor from $\Forb_{CP}(S^2)$. 
   \end{enumerate} 
\end{cor}

We note that any $B_m \boxempty P_3$, $Q_m \boxempty P_3$, $W_m \boxempty C_3$, and $X_m \boxempty C_3$ has crosscap number one for any $m \in \mathbb{N}$. Such graphs contain either a $K_{1,3} \boxempty P_3$ subgraph or a  $C_{3} \boxempty C_3$ subgraph, justifying that they are nonplanar.

For the next corollary, we refer the reader to Figure \ref{fig:Relevant graphs P3,C3 case} for visuals of $I$, $L$, $S$, and $R$. 

\begin{cor}
    \label{cor:MMNPP2}
    The set $\Forb_{CP}(N_1) = \{K_{2,3} \boxempty P_{2},  K_{4} \boxempty P_{2}, I \boxempty P_3, K_{1,5} \boxempty P_3, L \boxempty P_3, S \boxempty P_3, R \boxempty C_3, P_4 \boxempty K_{1,3}, K_{1,4} \boxempty K_{1,3} \}$ is a sufficient set of forbidden Cartesian product minors to classify projective planar Cartesian products. 
\end{cor}


\begin{proof}
     First, every $G \in \Forb_{CP}(N_1)$ is nonprojective planar -  one can find an element of $\Forb(N_1)$ from Appendix \ref{sec:Appendix} as a minor of any such $G$, and each such case was covered in an earlier section of this paper. We now show that $\Forb_{CP}(N_1)$ is a sufficient set of forbidden minors. We follow our classification results to justify this inclusion.  From Section \ref{sec:ClassificationII},  $G \boxempty P_2$ is projective planar if and only if  $G$ does not contain a $K_4$ or $K_{2,3}$ minor. Thus, $K_{4} \boxempty P_2$ and $K_{2,3} \boxempty P_2$ must be on our list.    In Section \ref{sec:ClassificationIII}, we see $K_{4} \boxempty P_2$, $K_{2,3} \boxempty P_2$, $I \boxempty P_3, K_{1,5} \boxempty P_3, L \boxempty P_3, S \boxempty P_3,  C_4 \boxempty C_3$ and  $R \boxempty C_3$ are used in the proofs of Theorem \ref{thm:N1 embedding P3 Case} and Theorem \ref{thm:N1 embedding C3 Case} to classify projective planar Cartesian products where one factor has exactly three vertices. The only one of these graphs not in $\Forb_{CP}(N_1)$ is $C_4 \boxempty C_3$, which contains a $K_{4} \boxempty P_2$ minor, and $K_{4} \boxempty P_{2} \in \Forb_{CP}(N_1)$. Finally, in Section \ref{sec:ClassificationIV}, we  used the facts that $P_4 \boxempty K_{1,3}$ or  $K_{1,4} \boxempty K_{1,3}$ are nonprojective planar in the proof of Theorem \ref{thm:N1 embedding 4 Vertex Case} to classify Cartesian products where each factor has at least four vertices.
\end{proof}

It is quite possible that every element of  $Forb_{CP}(N_1)$ is necessary. One could check necessity in a manner similar to the proof of Proposition \ref{prop:ForbiddenCPforplane}, though there would be several minor inclusions to rule out here. 
     


\section{Conclusion}

It is natural to ask if one can classify Cartesian products of graphs that embed in $S$, where $S$ is a surface other than the $2$-sphere or projective plane.  This could be a challenging task since $\Forb(S)$ is not known for any such surface, and in some cases, $\Forb(S)$ is known to have large subsets. For instance, there are at least $17,000$ distinct forbidden minimal minors for the torus; see \cite{MyWo218}.  However, only six of the $35$ elements in $Forb(N_1)$ where sufficient for our work. Thus, restricting to  Cartesian products  or Cartesian products with specific characteristics (perhaps  a fixed vertex connectivity for one factor) could make this analysis more tractable for the next lowest complexity surfaces, the torus and the Klein bottle. 

 \newpage

\appendix
\section{Minimal Forbidden Minors of the Projective Plane}
\label{sec:Appendix}
The following figure depicts the Forbidden minimal minors of the projective plane used throughout this paper. The labels follow those given by the list in Appendix A of  \cite{MoTh2001}.
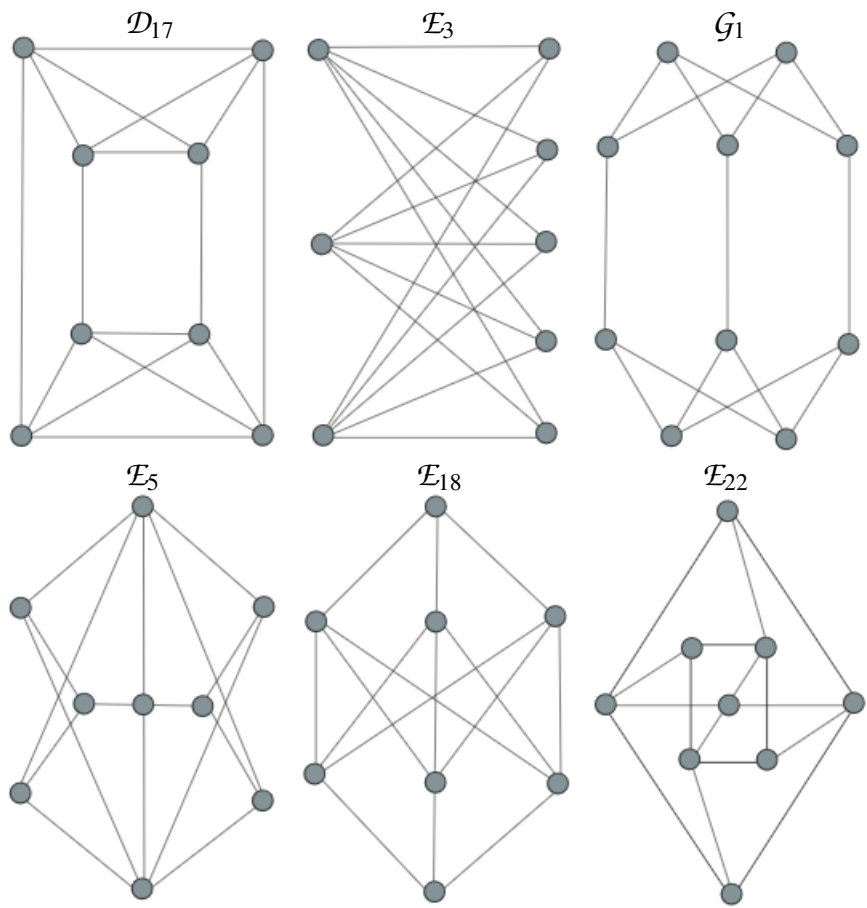
\begin{figure}[H]
    \centering
    \begin{overpic}[width=0.8\linewidth,grid=False]{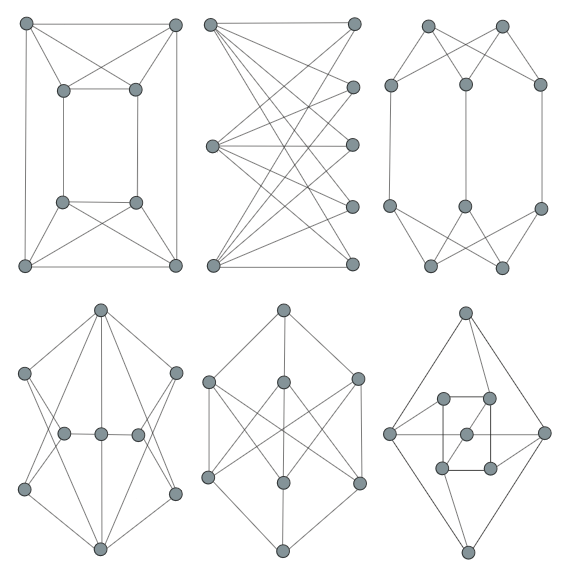}
        \put(15.5,97.5){$\mathcal D_{17}$}
        \put(47.5,97.5){$\mathcal E_3$}
        \put(79,97.5){$\mathcal G_1$}
        \put(15.5,48.5){$\mathcal E_5$}
        \put(47,48.5){$\mathcal E_{18}$}
        \put(78,48.5){$\mathcal E_{22}$}
    \end{overpic}
    \caption{The six graphs from $\Forb(N_1)$ that are sufficient to classify projective planar Cartesian products.}
    \label{fig:RelevantForbiddenMinorsofN1}
\end{figure}

\newpage


\bibliographystyle{hamsplain}
\bibliography{biblio}
\end{document}